\documentclass[sigconf]{acmart}

\title[Bounds for $q$-Recurrence Systems]{Denominator Bounds and Polynomial Solutions for Systems of
  $q$-Recurrences over $\K(t)$ for Constant $\K$} 
\date{2017}


\acmConference[ISSAC]{International Symposium on Symbolic and Algebraic Computation}{July 25-28, 2017}{Kai\-sers\-lau\-tern}

\author{Johannes Middeke}
\affiliation{%
  \institution{Johannes Kepler University}
  \department{Research Institute for Symbolic Computation (RISC)}
  \streetaddress{Altenbergerstraße 69}
  \city{Linz}
  \postcode{4040}
  \country{Austria}}
\email{jmiddeke@risc.jku.at}
\thanks{This work was supported by the Austrian Science Fund (FWF)
  grant SFB50 (F5009-N15).}

\theoremstyle{acmdefinition}
\newtheorem{algorithm}[theorem]{Algorithm}
\newtheorem{remark}[theorem]{Remark}

\renewcommand{\autoref}[2][Theorem]{\hyperref[#2]{#1~\ref*{#2}}}

\usepackage{xcolor}
\definecolor{TodoColour}{HTML}{A4B07B}
\definecolor{LinkColour}{HTML}{825419}
\definecolor{CiteColour}{HTML}{825419}
\definecolor{URLColour}{HTML}{825419}
\definecolor{EmphColour}{HTML}{878F1D}
\definecolor{SuperEmphColour}{HTML}{E69155}

\usepackage{listings}
\usepackage[bordercolor=TodoColour, 
backgroundcolor=TodoColour,  
linecolor=TodoColour, 
textsize=todosize]{todonotes}
\hypersetup{
  colorlinks, 
  urlcolor=URLColour,
  linkcolor=LinkColour,
  citecolor=CiteColour}

\newcommand{\becomes}{\leftarrow}

\DeclareMathOperator{\diag}{diag}

\usepackage{accents}
\newcommand{\dtilde}[1]{\accentset{\approx}{#1}}
\newcommand{\K}{\mathbb{K}}
\newcommand{\Q}{\mathbb{Q}}
\newcommand{\sbullet}{\bullet}

\renewcommand{\leq}{\leqslant}
\renewcommand{\geq}{\geqslant}
\renewcommand{\kappa}{\varkappa}
\renewcommand{\rho}{\varrho}

\newcommand{\qqtext}[1]{\qquad\text{#1}\qquad}
\newcommand{\qqand}{\qqtext{and}}

\newcommand{\ordinal}[1]{\raisebox{0.7ex}{\scriptsize#1}}

\newcommand{\Bigmid}{\Bigl.\;\Bigm|\;\Bigr.}

\newcommand{\Mat}[3]{#1^{#2\times#3}}
\newcommand{\MatGr}[2]{\operatorname{GL}_{#2}(#1)}
\newcommand{\CV}[2]{#1^{#2}}
\newcommand{\RV}[2]{#1^{1\times#2}}
\newcommand{\ID}[1][\relax]{\mathbf{1}\ifx#1\relax\relax\else_{#1}\fi}
\def\ZEROAUX#1,#2;{_{#1\times#2}}
\newcommand{\ZERO}[1][\relax]{\mathbf{0}\ifx#1\relax\relax\else\ZEROAUX#1;\fi}

\newcommand{\leqTOP}{\leq_\text{\scshape top}}

\newcommand{\RHS}{right hand side}
\newcommand{\LHS}{left hand side}

\newcommand{\ie}{that is}
\newcommand{\eg}{for example}
\newcommand{\wrt}{with respect to}
\newcommand{\PFD}{partial fraction decomposition}
\newcommand{\PNF}{Popov normal form}
\newcommand{\HNF}{Hermite normal form}
\newcommand{\JNF}{Jacobson normal form}
\newcommand{\GB}{Gr{\"o}bner basis}
\newcommand{\GBs}{Gr{\"o}bner bases}
\newcommand{\TOP}{term over position}
\newcommand{\POT}{position over term}

\usepackage{multicol}

\begin{document}

\begin{abstract}
  We consider systems $A_\ell(t) y(q^\ell t) + \ldots + A_0(t) y(t) =
  b(t)$ of higher order $q$-recurrence equations with rational
  coefficients.  We extend a method for finding a bound on the maximal
  power of $t$ in the denominator of arbitrary rational solutions
  $y(t)$ as well as a method for bounding the degree of polynomial
  solutions from the scalar case to the systems case. The approach is
  direct and does not rely on uncoupling or reduction to a first order
  system. Unlike in the scalar case this usually requires an initial
  transformation of the system.
\end{abstract}

\maketitle

\section{Introduction}\label{sec:intro}

Let $\K$ be a field of characteristic $0$, and let $q \in
\K\setminus\{0\}$ be not a root of unity. Moreover, let $t$ be an
indeterminate over $\K$. We define an automorphism $\sigma$ on the
rational function field $\K(t)$ by setting $\sigma(t) = q t$. We will
refer to $\sigma$ as the \emph{$q$-shift}. Note that with this
definition we have $\sigma(\alpha) = \alpha$ for all $\alpha \in
\K$. This is an easy case of a homogeneous or $\Pi$-extension in the
sense of \cite{Karr81}. In the following we will often let $\sigma$
act on (column) vectors in $\CV{\K(t)}m$. This is to be understood as
component-wise application of $\sigma$.

We consider the recurrence equation
\begin{equation}
  \label{eq:rec}
  A_s \sigma^s(y) + \ldots + A_1 \sigma(y) + A_0 y = b
  \tag{{\scshape rec}}
\end{equation}
where $A_0,\ldots,A_s \in \Mat{\K[t]}mm$ are square polynomial
matrices and $b \in \CV{\K[t]}m$ is a polynomial vector. We will
assume that the matrix $A = A_s \sigma^s + \ldots + A_0$ has full
(left row) rank over the operator ring $\K(t)[\sigma]$ (see
\autoref[Section]{sec:ore} for details). Note that any $q$-recurrence
system can be brought into this form using unimodular transformations
(see, \eg, \cite[Lem.~7.5]{myPhDthesis}; it is easy to verify that the
method discussed there works for general Ore operators -- see
\autoref[Section]{sec:ore} for the definition -- and not just
differential operators). We will briefly explain the idea at the end
of \autoref[Section]{sec:ore}.

\medskip

Generally, one major area where recurrence equations arise is symbolic
summation. See, \eg, \cite{Karr81}, \cite{SchneiderThesis},
\cite{AeqlB}, or \cite{ChyzakSalvy98}. In particular, note that
symbolic simplification of products in \cite{Karr81} and
\cite{SchneiderThesis} is done within so-called homogeneous or
$\Pi$-extension of which our field $\K(t)$ is a special case. Examples
which deal explicitly with $q$-hypergeometric sums are
\cite{PauleRiese:97} or \cite{Riese:03}.

\medskip

Our motivating goal is to find new algorithms to compute all rational
solutions of \eqref{eq:rec}; \ie, we want to know all $y \in
\CV{\K(t)}m$ such that the equation holds for that $y$.

A first step will be to determine a so-called \emph{denominator
  bound}: we are looking for a polynomial $d \in \K[t] \setminus
\{0\}$ such that $d y \in \CV{\K[t]}m$ for every possible rational
solution $y \in \CV{\K(t)}m$. There are comparatively easy ways to
determine all the factors in $d$ except for powers of $t$; cf., \eg,
\cite{AbramovBarkatou1998}, \cite{AbramovKhmelnov2012}, \cite{CPS:08},
or \cite{SchneiderMiddeke2017}. See \cite{Karr81}, \cite{Bron:00} or
\cite{SchneiderThesis} for a discussion on why this distinction is
necessary. Thus, in this work we will only concentrate on finding a
bound on the powers of $t$ in $d$ in a computationally cheap way. The
next step will then be to compute polynomial solutions of the
system. For this, we are computing a bound on the degree of polynomial
solutions $\tilde{y} \in \CV{\K[t]}m$. Once such a bound is
determined, a solution to~\eqref{eq:rec} can be found by making an
ansatz with unknown coefficients and solving the resulting linear
system over $\K$.

Methods for determining such a denominator bound (or universal
denominator as some authors prefer to call it), have already been
discussed in the literature. See, \eg, \cite{AbramovBarkatou1998},
\cite{Barkatou1999}, \cite{CPS:08}, or \cite{AbramovKhmelnov2012}.
Most of these approaches rely on the transformation of the system into
first order; the method in \cite{AbramovKhmelnov2012} seems to be the
only one which deals with systems directly (albeit for shift
recurrences and not the $q$-case; although it might be possible to
adapt the method using ideas from \cite{Abramov2002}). Another way of
determining a denominator bound would be to decouple the system using
cyclic vector methods, the Danilevski-Barkatou-Z{\"u}cher algorithm
(see \cite{Danilevski37}, \cite{Barkatou1993}, \cite{Zurcher94}, or
\cite{BCP13}), or Smith--Jacobson form computations.

In this paper, however, we want to avoid uncoupling and conversion to
first order systems as these often introduce an additional cost and
also tend to blur the structure of the input system. Instead we aim at
computing the denominator bound and degree bound directly from the
matrices of the higher order system~\eqref{eq:rec}. Our hope here is
that this will lead to more efficient algorithms.

\medskip

This paper is structured as follows: In \autoref[Section]{sec:bound},
we discuss how the powers of $t$ in the denominator can be bounded and
how we may determine a bound for the degree of polynomial
solutions. The results will require certain determinants not to be
identically zero. Since this of course cannot be guaranteed for an
arbitrary system, we need to develop an algorithm which transforms a
given system into an equivalent one where the method of
\autoref[Section]{sec:bound} is applicable. This is done in
\autoref[Section]{sec:det} which constitutes the major result of this
paper. The section also includes a comparison with competing methods.
In between, \autoref[Section]{sec:ore} recalls some basic facts about
operators and matrix normal forms which are needed for the
transformation.

\def\myT{{{{\it\lowercase{\Large t}}}}} 

\section{Bounding the Power of the Indeterminate and the Degree of
  Polynomial Solutions}\label{sec:bound}

Consider \eqref{eq:rec} where $A_0,\ldots,A_s$ are matrices in
$\Mat{\K[t]}mm$ and where $b \in \CV{\K[t]}m$ is a column vector. For
technical reasons the usefulness of which will become apparent in
\autoref[Section]{sec:det}, we consider the slightly modified system
\begin{equation}
  \label{eq:rec'}
  A_s \sigma^s(y) + \ldots + A_1 \sigma(y) + A_0 y = t^{-\nu} b
  \tag{{\scshape rec'}}  
\end{equation}
where $\nu \geq 0$, \ie, where we allow the \RHS\ to be a rational
vector but only with powers of $t$ in the denominator.

\medskip

We start by considering possible denominators. The following method is
folklore. Similar ideas can be found, \eg, in \cite{Karr81}. Let $y
\in \CV{\K(t)}m$ be an arbitrary solution of~\eqref{eq:rec'} and let
$g t^n \in \K[t]$ be the denominator of $y$ where $t \nmid g$. Using an
(entry-wise) \PFD, we may write $y = u/g + v/t^n$ where $u,v \in
\CV{\K[t]}m$ and where $v = 0$ or $t \nmid v$ (\ie, $t$ does not
divide every entry of $v$). In the first case, $n = 0$ is a candidate
for the power of $t$ in the degree bound. In the other case,
substituting this into~\eqref{eq:rec'} yields
\begin{equation*}
  \sum_{j=0}^s \frac{A_j \sigma^j(u)}{\sigma^j(g)}
  +
  \sum_{j=0}^s \frac{A_j \sigma^j(v)}{q^{n j} t^n}  
  =
  t^{-\nu} b
\end{equation*}
since $\sigma^j(t^n) = q^{n j} t^n$ for all $j$. We may transform the
equation into
\begin{multline*}
  t^{\nu - n} \Bigl(
  \prod_{j=0}^s \sigma^j(g)
  \Bigr)
  \sum_{j=0}^s q^{-n j} A_j \sigma^j(v)
  \\
  =
  \Bigl(
  \prod_{j=0}^s \sigma^j(g)
  \Bigr)
  b
  -
  \sum_{j=0}^s
  \Bigl(
  \prod_{k\neq j} \sigma^k(g)
  \Bigr)
  A_j \sigma^j(u).
\end{multline*}
Note that the \RHS\ is in $\CV{\K[x]}m$, \ie, a polynomial
vector. That means that also the \LHS\ must be polynomial, and this in
turn implies that either $\nu \geq n$ or else that $t$ divides
\begin{math}
  \bigl(                                                          
  \prod_{j=0}^s \sigma^j(g) 
  \bigr) 
  \sum_{j=0}^s q^{-n j} A_j \sigma^j(v).
\end{math}
Assume that the latter case holds. Then, since $t \nmid g$, we also
have $t \nmid \sigma^j(g)$ for all $j$ (or else $t \sim q^{-j} t =
\sigma^{-j}(t) \mid g$). Hence, $t$ cannot divide $\prod_{j=0}^s
\sigma^j(g)$. We obtain
\begin{equation*}
  t \Bigmid \sum_{j=0}^s q^{-n j} A_j \sigma^j(v).
\end{equation*}
In particular, the coefficient vector of $t^0$ in the sum must vanish.

Write now $A_j = A_{j\ell} t^\ell + \ldots + A_{j1} t + A_{j0}$ with
$A_{j0},\ldots, A_{j\ell} \in \Mat{\K}mm$ being constant matrices for
all $j$; and write $v = v_d t^d + \ldots + v_1 t + v_0$ with
$v_0,\ldots,v_d \in \CV{\K}m$. From the assumption that $t \nmid v$ we
know that in particular $v_0 \neq 0$. Using these representations in
the sum above, we find that the constant term is
\begin{equation*}
  0 = \Bigl( \sum_{j=0}^s q^{-nj} A_{j0} \Bigr) v_0.
\end{equation*}
Now, $\Lambda = \sum_{j=0}^s q^{-nj} A_{j0} \in \Mat{\K}mm$ is just a
matrix and since $v_0 \neq 0$ we see that $\Lambda$ must have a
non-trivial kernel. This implies that the determinant of $\Lambda$
vanishes. We may regard $\det \Lambda$ as a polynomial expression in
$q^{-n}$. Therefore, assuming that $\det \Lambda$ does not identically
vanish, we can obtain candidates for $n$ by checking if it has roots
of the form $q^{-n}$.

This observation motivates the following definition.

\begin{definition}[$t$-tail regular]\label{def:t-tail.reg}
  We call the system~\eqref{eq:rec'} \emph{$t$-tail regular} if
  \begin{math}
    \lambda = \det \bigl( \sum_{j=0}^s A_{j0} x^j \bigr) \in \K[x]
  \end{math}
  is not the zero polynomial.
\end{definition}


Summarising the argumentation so far, we obtain
\autoref[Theorem]{thm:deg.bound} below. It remains to deal with the
case of systems which are not $t$-tail regular. We will do so in
\autoref[Section]{sec:det}.

\begin{theorem}\label{thm:deg.bound}
  Assume that the system~\eqref{eq:rec'} is $t$-tail regular. If $y
  \in \CV{\K(t)}m$ is a rational solution of~\eqref{eq:rec'} and if
  $t^n$ divides the denominator of $y$, then
  \begin{enumerate}
  \item $n = 0$,
  \item or $n \leq \nu$,
  \item or $q^{-n}$ is a root of 
    \begin{math}
      \lambda = \det \bigl( \sum_{j=0}^s A_{j0} x^j \bigr) \in \K[x].
    \end{math}
  \end{enumerate}
  Consequently, we obtain a bound on $n$ by taking the maximum of
  those values. \qed
\end{theorem}

If $\K$ happens to be of the form $\mathbb{F}(q)$ where $\mathbb{F}$
is a computable field and $q$ is transcendental over $\mathbb{F}$,
then the method of \cite[Ex.~1]{AbramovPaulePetkovsec} can be applied
to find all roots of $\lambda$ of the required form. In the case that
$\K = \Q$, we can use the method of
\cite[Sect.~3.3]{BauerPetkovsek1999} for that.

\begin{example}\label{ex:bound}
  Let $\K = \Q$ and $q = 2$. We consider the system
  \begin{multline*}
    \begin{pmatrix}
      8 & 0 \\
      8 & 0
    \end{pmatrix}
    \begin{pmatrix}
      y_1(4 t) \\
      y_2(4 t) \\
    \end{pmatrix}
    + 
    \begin{pmatrix}
      -16 t + 4 & 8 \\
      -16 t^2 + 16 t - 12 & 8 
    \end{pmatrix}
    \begin{pmatrix}
      y_1(2 t) \\
      y_2(2 t) \\
    \end{pmatrix}
    \\
    + 
    \begin{pmatrix}
      16 t - 4 & -8 t^3 - 1 \\
      16 t^2 - 8 t + 4 & -8 t^4 - 1
    \end{pmatrix}
    \begin{pmatrix}
      y_1(t) \\
      y_2(t) \\
    \end{pmatrix}
    = 0.
  \end{multline*}
  Here, $\nu = 0$. Computing the polynomial $\lambda$ yields
  \begin{equation*}
    \lambda =
    \det \begin{pmatrix}
      8 x^2 + 4 x - 4 & 8 x - 1 \\
      8 x^2 - 12 x + 4 & 8 x - 1 \\
    \end{pmatrix}
    = 128 x^2 - 80 x + 8.
  \end{equation*}
  The roots of $\lambda$ are $1/2 = q^{-1}$ and $1/8 = q^{-3}$. That
  yields the upper bound $3$ for $n$. Note that this fits the actual
  solutions well which are generated by
  \begin{equation*}
    \begin{pmatrix}
      1 \\ t^{-3} 
    \end{pmatrix}
    \qqand
    \begin{pmatrix}
      t^{-1} \\ t^{-3}
    \end{pmatrix}.
  \end{equation*}
  (We can easily check that the two vectors are indeed solutions; and
  -- after row reduction -- \cite[Thm.~6]{AbramovBarkatou2014} gives
  the dimension of the solution space as $2$.)
\end{example}

\medskip

We turn our attention to polynomial solutions now. The approach is
rather similar to the degree bounding and has been discussed in the
literature before, \eg, in
\cite[Sect.~2]{AbramovPaulePetkovsec}. Assume that $y \in \CV{\K[t]}m$
is a solution of~\eqref{eq:rec'}. We make an ansatz
\begin{equation*}
  y = y_n t^n + \ldots + y_1 t + y_0
\end{equation*}
where $y_0, \ldots, y_n \in \CV{\K}m$. We assume that $y_n \neq 0$,
\ie, that $n$ is the degree of $y$. In addition, let $b = b_\kappa
t^\kappa + \ldots + b_1 t + b_0$ with $b_0, \ldots, b_\kappa \in
\CV{\K}m$; and write once more
\begin{math}
  A_j = A_{j\ell} t^\ell + \ldots + A_{j1} t + A_{j0}
\end{math}
with $A_{j0},\ldots,A_{j\ell} \in \Mat{\K}mm$ for all $j$. We assume
that $b_\kappa \neq 0$ and that $A_{j\ell} \neq \ZERO$ for at least
one $j$.  Substituting all of this into~\eqref{eq:rec'} gives
\begin{multline*}
  \sum_{j=0}^s 
  \Bigl( \sum_{i=0}^\ell A_{ji} t^i \Bigr) 
  \sigma^j\Bigl( \sum_{k=0}^n y_k t^k \Bigr)
  \\
  = 
  \sum_{j=0}^s 
  \sum_{i=0}^\ell 
  \sum_{k=0}^n
  q^{jk} t^{i+k} A_{ji} y_k
  = 
  t^{-\nu} \sum_{j=0}^\kappa b_j t^j.
\end{multline*}
If the \RHS\ contains negative powers of $t$, then there cannot be a
polynomial solution. Else, the degree (in $t$) of the \LHS\ is at most
$\ell + n$. We compute the coefficient of $t^{\ell + n}$ on both
sides. This yields (since $i = \ell$ and $k = n$)
\begin{equation*}
  \Bigl( \sum_{j=0}^s A_{j\ell} q^{n j} \Bigr) y_n = b_{\ell+n+\nu}.
\end{equation*}
There are two cases:
\begin{enumerate}
\item Either $\kappa \geq \ell + n + \nu$ (\ie, $n \leq \kappa - \nu -
  \ell$),
\item or $\kappa < \ell + n + \nu$.
\end{enumerate}
In the first case, we do already have a bound on $n$ (namely $\kappa -
\nu - \ell$). In the second case, the \RHS\ of the equation for the
coefficient of $t^{\ell+n}$ is zero. Just as for the denominator
bound, we see that the \LHS\ contains a matrix with a non-trivial
kernel. Again, its determinant is a polynomial expression---this time
in $q^n$. We make an analogous definition.

\begin{definition}[$t$-head regular]
  We call the system~\eqref{eq:rec'} \emph{$t$-head regular} if the
  polynomial
  \begin{math}
    \rho = \det \bigl( \sum_{j=0}^s A_{j\ell} x^j \bigr)
    \in \K[x]
  \end{math}
  is not identically zero.
\end{definition}


The computations above have show that for a $t$-head regular system
$q^n$ is a root of $\rho$. Thus, we can again obtain all the
candidates for the degree $n$ of $y$ by computing the roots of
$\rho$. The following theorem summarises these findings.

\begin{theorem}\label{thm:deg}
  Assume that the system~\eqref{eq:rec'} is $t$-head regular. Let
  $\ell$ be the maximum of $\deg_t A_0, \ldots, \deg_t A_s$ and let
  $\kappa = \deg_t b$. If $y \in \CV{\K[t]}m$ is a polynomial solution
  of degree $n$, then
  \begin{enumerate}
  \item either $n \leq \kappa - \nu - \ell$,
  \item or $q^n$ is a root of
    \begin{math}
      \rho = \det \bigl( \sum_{j=0}^s A_{j\ell} x^j \bigr)
      \in \K[x].
    \end{math}
  \end{enumerate}
  As in \autoref[Theorem]{thm:deg.bound}, we obtain a finite number of
  candidates for $n$.
\end{theorem}

\begin{example}\label{ex:deg}
  We continue \autoref[Example]{ex:bound}. We have already seen that
  rational solutions have the shape $y = t^{-3} z$ for some $z \in
  \CV{\K[t]}m$.\footnote{We did not explicitly show that the
    denominator of $y$ does not have any other factors apart from $t$
    -- however, this is the case here since we do know the complete
    solution space.} Substituting this into the equation and clearing
  denominators yields
  \begin{multline*}
    \begin{pmatrix}
      1 & 0 \\
      1 & 0 
    \end{pmatrix}
    \begin{pmatrix}
      z_1(4 t) \\
      z_2(4 t)
    \end{pmatrix}
    +
    \begin{pmatrix}
      4 - 16 t & 8 \\
      12 + 16 t - 16 t^2 & 8 
    \end{pmatrix}
    \begin{pmatrix}
      z_1(2 t) \\
      z_2(2 t)
    \end{pmatrix}
    \\
    +
    \begin{pmatrix}
      128 t - 32 & -64 t^3 - 8 \\
      128 t^2 - 64 t + 32 & - 64 t^4 - 8
    \end{pmatrix}
    \begin{pmatrix}
      z_1(t) \\
      z_2(t)
    \end{pmatrix}
    = 0.
  \end{multline*}
  The maximal degree in $t$ is $\ell = 4$ and thus, we have
  \begin{equation*}
    \rho = \det
    \begin{pmatrix}
      0 & 0 \\
      0 & -64
    \end{pmatrix}
    = 0.
  \end{equation*}
  Thus, \autoref[Theorem]{thm:deg} is not applicable here. We will
  revisit the example in \autoref[Section]{sec:det}.
\end{example}

\section{Ore Operators and Matrix Normal Forms}\label{sec:ore}

This section is devoted to introducing some concepts which are
necessary in order to deal with the cases $\lambda = 0$ or $\rho =
0$. This section introduces the algebraic model which we are going to
employ as well as some necessary background information and other
results.

\medskip

In \autoref[Section]{sec:det}, we will try to transform the
system~\eqref{eq:rec'} into an equivalent form such that the new
$\lambda$ or $\rho$ becomes non-zero. During this, we are allowed to
use the following transformations:
\begin{enumerate}
\item multiply a row of the system by a rational function in $\K(t)$;
\item add a $\K(t)$-linear combination of shifted versions of one row
  to another row.
\end{enumerate}
We will apply the transformations in such a way that the coefficients
on the \LHS\ of~\eqref{eq:rec'} will always be polynomial
matrices. However, the \RHS\ might contain rational functions in
$t$. 

It will be convenient to express the transformations using the
language of \emph{Ore operators}. Roughly speaking, these are a class
of non-commutative polynomials. They are named after their first
introduction by \O.~Ore in \cite{Ore33}; a first interpretation of
them as operators seems to be~\cite{JacPLT}. Ore operators have been
used to model differential or difference equations (see \cite{BP96}
for an introduction) and have been applied to problems in symbolic
summation (cf.\ \cite{ChyzakSalvy98}). Implementations exist for
\textsc{Maple} (\eg, \cite{AbramovLeLi2005}) and \textsc{Mathematica}
(\eg, \cite{KoutschanPhD}).

We will only require a special case where the derivative is the zero
map. These rings are sometimes also referred to as \emph{skew
  polynomials}. For us, this is simply the (non-commutative) ring
\begin{equation*}
  \K[t,\sigma] 
  = \{ \alpha_d(t) \sigma^d + \ldots \alpha_1(t) \sigma + \alpha_0 
  \mid 
  \alpha_0, \ldots, \alpha_d \in \K[t] \}
\end{equation*}
which consist of polynomial expressions in $\sigma$ with coefficients
being polynomials in $t$. The addition in this ring is the same as for
regular polynomials. The multiplication, however, is governed by the
\emph{commutation rule}
\begin{equation*}
  \sigma t = q t \sigma
\end{equation*}
(and consequently $\sigma t^{-1} = q^{-1} t^{-1} \sigma$ in the
extensions which we introduce below). We will usually refer to the
elements of $R$ as \emph{operators}. These operators act on $\K(t)$ in
the usual way: Denoting the action of $a \in \K[t,\sigma]$ on $y \in
\K(t)$ by $a \sbullet y$ we have $t\sbullet y(t) = t y(t)$ and
$\sigma\sbullet y(t) = y(q t)$.

We will also need to consider a related ring where we allow rational
functions in $t$ in the coefficients of our operators. The
corresponding ring $\K(t)[\sigma]$ may be constructed from
$\K[t,\sigma]$ since the non-zero polynomials in $t$ form a so-called
(left or right) Ore set in $\K[t,\sigma]$ (for a definition see, \eg,
\cite[p.~177]{cohnRT}). We have $\K[t,\sigma] \subseteq
\K(t)[\sigma]$. The action of $\K[t,\sigma]$ on $\K(t)$ can easily be
extended to an action of $\K(t)[\sigma]$.

We may now associate the system~\eqref{eq:rec'} with a matrix $A \in
\Mat{\K[t,\sigma]}mm$ writing it as
\begin{equation*}
  \label{eq:2b}
  A \sbullet y = t^{-\nu} b.
\end{equation*}
Doing so, we may identify the transformations above as multiplication
of the system by a certain matrix from the left. More precisely, each
of the allowed transformations correspond to multiplication by a
\emph{unimodular} matrix $Q \in \Mat{\K(t)[\sigma]}mm$, \ie, a matrix
which has an inverse $Q^{-1} \in \Mat{\K(t)[\sigma]}mm$ in the same
matrix ring. We denote the set of unimodular matrices by
$\MatGr{\K(t)[\sigma]}m$. Note that $Q$ transforms the
system~\eqref{eq:rec'} into
\begin{equation*}
  ( Q A ) \sbullet y
  = Q \sbullet t^{-\nu} b
\end{equation*}
where both the left and the right hand side are modified. It is easy
to see that the solutions of the original and the transformed system
are exactly the same if $Q$ is unimodular. We will call the two
systems (and by extension also the matrices $A$ and $Q A$)
\emph{equivalent}.

We would like to point out that being unimodular is a stronger
property than merely having full (left row) rank where we understand
the rank as the maximal number of linearly independent rows over
$\K(t)[\sigma]$ -- see also~\cite[Def.~2.1,
Thm.~A.2]{BeckermannChengLabahn2006}. For instance, the $1$-by-$1$
matrix $(\sigma) \in \Mat{\K(t)[\sigma]}11$ has full rank, but no
inverse in $\Mat{\K(t)[\sigma]}11$. We will call a matrix of full rank
\emph{regular} in order to distinguish this case.

\medskip

Before we are able to propose an algorithm for transforming the
system~\eqref{eq:rec'} into an equivalent system which is $t$-tail
regular or $t$-head regular, we need to introduce one more
concept. First, we would like to point out that the subring
$\K[\sigma]$ of $\K[t,\sigma]$ may actually be regarded as a
commutative polynomial ring because $\sigma$ commutes with all
elements of $\K$. Therefore, it makes sense to look at matrix normal
forms for polynomial matrices. We will again use the terms unimodular
and regular, this time referring to matrices which have an inverse in
$\Mat{\K[\sigma]}mm$ or which have full rank over $\K[\sigma]$,
respectively.

The normal form which we are going to employ here is the so-called
\emph{\PNF}. The most succinct way to describe it is using \GBs\ for
modules. See \cite{myPhDthesis} for a detailed explanation; or see
\cite{Villard96} or Popov's original paper \cite{Popov72} for a more
traditional definition. See, \eg, \cite{AdamsLoustaunau1994} for more
details on \GBs\ over modules. We introduce the \emph{\TOP\ ordering}
on the polynomial module $\RV{\K[\sigma]}m$ by saying that $e_i
\sigma^a \leqTOP e_j \sigma^b$ if $a < b$ or $a = b$ and $i \geq
j$. Here, $e_i$ denotes the $i$\ordinal{th}\ unit vector. We say that
a matrix is in \PNF\ if its rows form a minimal \GB\ \wrt\
$\leqTOP$. It is possible to prove that for every matrix $M \in
\Mat{\K[\sigma]}mm$ there exists a unimodular matrix $Q \in
\MatGr{\K[\sigma]}m$ over $\K[\sigma]$ such that\footnote{That is,
  here unimodularity means that the inverse of $Q$ is a polynomial
  matrix $Q^{-1} \in \Mat{\K[\sigma]}mm$.}
\begin{equation*}
  Q M =
  \begin{pmatrix}
    P \\ \ZERO
  \end{pmatrix}
\end{equation*}
where $P \in \Mat{\K[\sigma]}{\ell}m$ is in \PNF. It can further be
shown that a matrix in \PNF\ has maximal row rank. Also, using \GB\
division (for each row of $C$) we can write every matrix $C \in
\Mat{\K[\sigma]}{\ell}m$ as
\begin{math}
  C = X P + Y
\end{math}
where $X \in \Mat{\K[\sigma]}\ell\ell$ and where $Y \in
\Mat{\K[\sigma]}\ell{m}$ has no rows which are reducible by the \PNF\
$P$ (again in the \GB\ sense). This offers an algorithmic way to
determine whether $C$ is a left multiple of $P$. We are going to need
this property in \autoref[Algorithm]{alg:lambda} and
\autoref[Algorithm]{alg:rho} below.

We would like to point out, that despite its definition in terms of
\GBs, the \PNF\ can easily be computed in polynomial time. In fact,
even the naive method which is based on row reduction will be
polynomial (cf.\ \cite[Lem.~5.14]{myPhDthesis}). See, \eg,
\cite[Cor.~6.1]{BeckermannLabahnVillard2006} for a faster method.

It is possible to use other row equivalent forms instead of the \PNF;
as long as they have a similar division property. One possible example
is the \HNF\ which is a \GB\ \wrt\ the \POT\ ordering (see, \eg,
\cite{myPhDthesis} for a details on why the \HNF\ is a \GB). However,
to our best knowledge, with current methods computing a \PNF\ can be
done more efficiently than computing a \HNF.

\medskip

We briefly recall the reasoning of \cite[Lem~7.5]{myPhDthesis} to
explain how the \PNF\ can also be used to remove redundancies from a
system: Starting with a general matrix $A \in \Mat{\K[t,\sigma]}mn$ --
\ie, with a matrix which is not necessarily square or of full row rank
-- we first compute the \emph{column} \PNF\ of $A$. It is defined
analogously but for right modules of column vectors instead of left
modules of row vectors. Moreover, algorithms to compute the (row)
\PNF\ and the corresponding transformation matrix can easily be
translated to the column \HNF. We obtain a representation
\begin{equation*}
  A U =
  \begin{pmatrix}
    \tilde{A} & \ZERO
  \end{pmatrix}
\end{equation*}
where $U \in \MatGr{\K(t)[\sigma]}n$ and $\tilde{A}$ has full column
rank. Next, we compute the row \PNF\ of $A U$ which yields
\begin{equation*}
  Q A U =
  \begin{pmatrix}
    \dtilde{A} & \ZERO \\
    \ZERO & \ZERO
  \end{pmatrix}
\end{equation*}
where $Q \in \MatGr{\K(t)[\sigma]}m$ and $\dtilde{A}$ has full row
rank. Since the row transformations do not change the column rank and
since (left) row and (right) column rank coincide by
\cite[Thm.~8.1.1]{cohn}, we conclude that $\dtilde{A}$ must be square
and of full row rank. 

Now, as $Q$ and $U$ are both unimodular, the system $A \sbullet x = b$
has a solution if and only if the system $\dtilde{A} \sbullet y = c$
has a solution and $w = 0$ where
\begin{equation*}
  x = Q
  \begin{pmatrix}
    y \\ z
  \end{pmatrix}
  \qqand
  Q \sbullet b =
  \begin{pmatrix}
    c \\ w
  \end{pmatrix}
\end{equation*}
(with the blocks matching those of $Q A U$). Thus, either the
compatibility condition $w = 0$ does not hold and we know that the
system is not solvable; or it suffices to concentrate on the system
$\dtilde{A} \sbullet y = c$ which is of the correct form for the
method presented in this paper. After solutions $y$ have been found,
we can easily translate them into solutions $x$ of the original system
using the matrix $Q$. (The vector $z$ does not have any conditions
imposed on it and contributes free variables to the solution $x$.)

It is not necessary to employ the \PNF\ in order to remove the
redundancies of the system. Any pair of rank revealing column and row
transformation would be sufficient. For instance, we could use the
\HNF. Since we do not need the division property, also row and column
reduction (see \cite{BeckermannChengLabahn2006} for a definition and
an efficient algorithm) or EG-eliminations (see \cite{Abramov:99})
could be used.

\section{Desingularising the  Determinants}\label{sec:det}

With the notation and concepts from \autoref[Section]{sec:ore}, we are
finally prepared to deal with the case of identically vanishing
determinants in \autoref[Section]{sec:bound}. This section here will
introduce an algorithm for transforming any system into an equivalent
one (using elementary row transformations over $\K(t)[\sigma]$) where
$\lambda$ or $\rho$ do not vanish. This constitutes the main result of
this paper.

\medskip

We write system~\eqref{eq:rec'} in operator form $A \sbullet y =
t^{-\nu} b$ where $A \in \Mat{\K[t,\sigma]}mm$ is an operator matrix
as explained in \autoref[Section]{sec:ore}. We express $A$ as
\begin{equation*}
  A = t^\ell \tilde{A}_\ell + \ldots + t \tilde{A}_1 + \tilde{A}_0
\end{equation*}
where $\tilde{A}_0, \ldots, \tilde{A}_\ell \in \Mat{\K[\sigma]}mm$ are
matrices in $\sigma$ alone. We will call $\tilde{A}_0$ the
\emph{$t$-trailing matrix} of $A$ and -- assuming that $\tilde{A}_\ell
\neq \ZERO$ -- we will call $\tilde{A}_\ell$ the \emph{$t$-leading
  matrix of $A$}. It is easy to see that (using the notation of
\autoref[Section]{sec:bound}) we have
\begin{equation*}
  \tilde{A}_0 = \sum_{j=0}^s A_{j0} \sigma^j
  \qqand
  \tilde{A}_\ell = \sum_{j=0}^s A_{j\ell} \sigma^j.
\end{equation*}
In other words, $\lambda$ and $\rho$ are the same as the determinants
of $\tilde{A}_0$ and $\tilde{A}_\ell$ -- except that we used $x$
instead of $\sigma$ as the name for the variable. (As explained in
\autoref[Section]{sec:ore}, $\K[\sigma]$ can be interpreted as
univariate commutative polynomial ring.) Thus, the task of
transforming the system into an equivalent one (using a unimodular
multiplier) with $\lambda$ or $\rho$ non-zero can be equivalently
described as the task of having the $t$-trailing or the $t$-leading
matrix being regular.

The definition of the $t$-trailing matrix in particular is also the
reason why we allowed a denominator for the \RHS\
of~\eqref{eq:rec'}. Should $\tilde{A}_0$ be the zero matrix, then we
simply divide the entire equation by a suitable power of $t$ in order
rectify that problem.

\medskip

There are similar works which consider recurrence systems and place
requirements on certain leading or trailing matrices. For instance,
see the method in \cite{AbramovKhmelnov2012} or the concept of
\emph{strong row-reduction} in
\cite[Def.~4]{AbramovBarkatou2014}. There is, however, an essential
difference between those approaches and the one presented here: We
consider the leading and trailing matrices \wrt\ $t$ as the main
variable while those other approaches consider $\sigma$ to be the main
variable (and are consequently dealing with leading and trailing
matrices that are in $\K(t)$).

The choice of variable does make a big difference. While we can easily
do a non-commutative version of row-reduction (\ie, working over
$\K(t)[\sigma]$) in order to work on the leading coefficient \wrt\
$\sigma$, say; simply switching the variables would catapult us into
$\K(\sigma)[t]$. The later ring, however, contains arbitrary
denominators in $\sigma$ and therefore does not have a natural action
on $\K(t)$ which extends the action of $\K[t,\sigma]$. (To see that
let $p \in \K[\sigma]$ and $f \in \K(t)$. If we want to compute
$p^{-1}\bullet f$, then we have to find $g \in \K(t)$ such that
$p\bullet g = f$. However, not all equations of that form have
rational solutions.)

We will thus need to develop a different approach which does not
require division by expressions in $\sigma$.

\smallskip

Also note that our algorithm bears some similarities to
\emph{EG-eliminations} as described in \cite{Abramov:99}. In
particular, the idea to reduce the trailing matrix $\tilde{A}_0$ and
to shift down coefficients from the higher matrices whenever a row of
$\tilde{A}_0$ becomes zero are the same. The difference is again the
choice of the main variable. The transformations used during
EG-eliminations are unimodular over $\K(t)[\sigma, \sigma^{-1}]$ (the
Laurent skew polynomials\footnote{These are well-defined since the
  powers of $\sigma$ form an Ore set in $\K(t)[\sigma]$ -- see, \eg,
  \cite[Chptr.~5]{cohnRT}}) and do thus not alter the solutions;
however switching the main variable does again expose us to arbitrary
fractions in $\sigma$.

The need to keep all transformations unimodular over $\K(t)[\sigma]$
also means that we cannot proceed row by row as \cite{Abramov:99} does
where the $\sigma$-trailing or -leading matrix is brought into
trapezoidal form one row at a time. Instead we always have to consider
the entire $t$-trailing or -leading matrix for each elimination
step. In particular, we cannot force the ``$t$-width'' (as defined in
\cite{Abramov:99}) of the lower rows to decrease.

\medskip

We are going to deal with the $t$-trailing matrix first. Below we
state an algorithm which we claim will transform the system into the
desired shape. Note that while we do not explicitly compute the
transformation matrix $Q$, this could be easily done by applying all
the transformations done to $A$ in parallel to the identity matrix
$\ID[m]$. Alternatively, we may also apply them directly to the \RHS\
of~\eqref{eq:rec'}.

\begin{algorithm}\label{alg:lambda}~
  \begin{description}
  \item[Input:] A regular matrix $A \in \Mat{\K[t,\sigma]}mm$.
  \item[Output:] An equivalent matrix $B \in \Mat{\K[t,\sigma]}mm$
such that the $t$-trailing matrix of $B$ is regular.
  \item[Procedure:]
    \begin{enumerate}
    \item\label{outer} Compute the $t$-trailing matrix $\tilde{A}_0$
      of $A$. If $\det \tilde{A}_0 \neq 0$, then return $A$.
    \item\label{X} Else, compute the \PNF\ of $\tilde{A}_0$ and let $X
      \in \MatGr{\K[\sigma]}m$ be the corresponding transformation
      matrix and let $r$ be the rank of $\tilde{A}_0$.\footnote{That
        is, the number of nonzero rows.} Set $A \becomes X A$ and
      write the new trailing matrix of $A$ as
      \begin{equation*} \tilde{A}_0 =
        \begin{pmatrix} \tilde{B}_0 \\ \ZERO
        \end{pmatrix}
      \end{equation*} where $\tilde{B}_0 \in \Mat{\K[\sigma]}rm$ is in
      \PNF.
    \item\label{inner} Let $T = \diag(\ID[r], t^{-1} \ID[m-r])$ and
      set $A \becomes T A$.\footnote{This shifts the lower rows of $A$
        by $t^{-1}$.} The new trailing matrix is now
      \begin{equation*} \tilde{A}_0 =
        \begin{pmatrix} \tilde{B}_0 \\ \tilde{C}_0
        \end{pmatrix}
      \end{equation*} with the same $\tilde{B}_0$ as before and some
      $\tilde{C}_0 \in \Mat{\K[\sigma]}{(m-r)}m$.
    \item\label{GB} Use \GBs\ reduction of $\tilde{C}_0$ by
      $\tilde{B}_0$ trying to eliminate all the rows in
      $\tilde{C}_0$. Write the result as $\tilde{C}_0 = X \tilde{B}_0
      + Y$ with a matrix $X \in \Mat{\K[\sigma]}rr$ and a matrix $Y
      \in \Mat{\K[\sigma]}{(m-r)}m$. Let
      \begin{equation*} Q =
        \begin{pmatrix} \ID[r] & \ZERO \\ -X & \ID[m-r]
        \end{pmatrix}
      \end{equation*} and set $A \becomes Q A$.
    \item If $Y = \ZERO$, then continue the inner loop with
      \autoref[Step]{inner}. Else, go back to the outer loop in
      \autoref[Step]{outer}.
    \end{enumerate}
  \end{description}
\end{algorithm}

We would like to point out that for any matrix $X \in
\Mat{\K[\sigma]}mm$ we have $X t = t \overline{X}$ for some
$\overline{X} \in \Mat{\K[\sigma]}mm$ since $\sigma$ is the
$q$-shift. Thus, in \autoref[Step]{X} of
\autoref[Algorithm]{alg:lambda} the product $X A$ has the form $t^\ell
X^{(\ell)} \tilde{A}_\ell + \ldots + t X^{(1)} \tilde{A}_1 + X
\tilde{A}_0$ for some matrices $X^{(1)}, \ldots, X^{(\ell)} \in
\Mat{\K[\sigma]}mm$. In other words, $X$ acts on the $t$-trailing
matrix of $A$ without interference from the other parts. Consequently,
if after applying the transformation by $X$ the lower rows of the
$t$-trailing matrix $\tilde{A}_0$ become zero, this means that the
lower rows of the entire matrix $A$ must be divisible by $t$. This
explains why the transformation by $T$ in \autoref[Step]{inner} will
still result in a matrix $A \in \Mat{\K[t,\sigma]}mm$. A similar
reasoning as in \autoref[Step]{X} also holds true for
\autoref[Step]{GB}: again the transformation of $A$ by $Q$ acts on the
$t$-trailing matrix without disturbance from the other
coefficients. Its effect is to replace the lower block $\tilde{C}_0$
of $\tilde{A}_0$ by $Y$. 

Note that this reasoning only works because of our special choice of
$\sigma$ and would not be possible with more general automorphisms. In
particular, the algorithm is not applicable for the normal shift case.

\begin{theorem}\label{thm:lambda}
  \autoref[Algorithm]{alg:lambda} is correct and terminates.
\end{theorem}

Before we can prove the theorem, we state two simple remarks:

\begin{remark}\label{rem:trans}
  As long as the inner loop runs, we always consider the trailing
  matrix
  \begin{equation*}
    A_0 =
    \begin{pmatrix}
      B_0 \\ C_0
    \end{pmatrix}
  \end{equation*}
  and compute a matrix $X$ (if possible) such that $C_0 = X B_0$. (The
  inner loop terminates if that is not possible any more.) We then
  apply the following transformations to $A$
  \begin{multline*}
    A =
    \begin{pmatrix}
      B \\ C
    \end{pmatrix}
    \xrightarrow{\text{reduction}}
    \begin{pmatrix}
      \ID & \ZERO \\
      -X & \ID
    \end{pmatrix}
    \begin{pmatrix}
      B \\ C
    \end{pmatrix}
    \\
    \xrightarrow{\text{shift}}
    \begin{pmatrix}
      \ID & \ZERO \\
      \ZERO & t^{-1} \ID
    \end{pmatrix}
    \begin{pmatrix}
      \ID & \ZERO \\
      -X & \ID
    \end{pmatrix}
    \begin{pmatrix}
      B \\ C
    \end{pmatrix}
    =
    \begin{pmatrix}
      \ID & \ZERO \\
      -t^{-1} X & t^{-1} \ID
    \end{pmatrix}
    \begin{pmatrix}
      B \\ C
    \end{pmatrix}.
  \end{multline*}
  Thus, we see that the transformation matrix for a single step has
  the form
  \begin{equation*}
    \begin{pmatrix}
      \ID & \ZERO \\
      - t^{-\nu} X & t^{-\nu}
    \end{pmatrix}
  \end{equation*}
  for some $\nu \geq 1$ (where we write $t^{-\nu}$ in the lower right
  block instead of $t^{-\nu} \ID$). Moreover,
  \begin{equation*}
    \begin{pmatrix}
      \ID & \ZERO \\
      - t^{-\mu} X & t^{-\mu}
    \end{pmatrix}    
    \begin{pmatrix}
      \ID & \ZERO \\
      - t^{-\nu} \tilde{X} & t^{-\nu}
    \end{pmatrix}    
    =
    \begin{pmatrix}
      \ID & \ZERO \\
      - t^{-(\mu+\nu)}(t^\nu X +  \tilde{X}) & t^{-(\mu+\nu)}
    \end{pmatrix}
  \end{equation*}
  for all $\mu,\nu \geq 0$ and matrices $X, \tilde{X} \in
  \Mat{\K[t,\sigma]}{(m-r)}r$. Thus, we easily see that all
  transformation matrices in the inner loop must be of this shape.
\end{remark}

\begin{remark}\label{rem:inverse}
  Let $A \in \Mat{\K[t,\sigma]}mm$ have full (left) row rank. We want
  to prove that $A$ has a (two-sided) inverse in the quotient skew
  field $\K(t,\sigma)$. We first remark that we can embed
  $\K[t,\sigma]$ into the non-commutative Euclidean domain
  $\K(t)[\sigma]$. Over $\K(t)[\sigma]$, we can compute the \JNF\
  \begin{equation*}
    N = \diag(e_1,e_2,\ldots,e_m) = S A T
  \end{equation*}
  of $A$ where $S, T \in \MatGr{\K(t)[\sigma]}m$ are unimodular (see,
  \eg,~\cite[Thm.~8.1.1]{cohn}). The diagonal entries $e_1,\ldots,e_m$
  cannot be zero since otherwise $A$ did not have full row rank. Thus,
  the inverse $N^{-1} = \diag(e_1^{-1},\ldots,e_m^{-1}) \in
  \Mat{\K(t,\sigma)}mm$ of $N$ exists. We obtain
  \begin{equation*}
    \ID = N^{-1} S A T, 
    \qqtext{or, equivalently,}
    S^{-1} N T^{-1} = A.
  \end{equation*}
  Since $S^{-1} N T^{-1}$ is the product of (from both sides)
  invertible matrices, we conclude that $A$ is invertible.
\end{remark}

\begin{proof}[Proof of {\autoref[Theorem]{thm:lambda}}]
  If the algorithm terminates, then the trailing matrix $\tilde{A}_0$
  is in \PNF\ which implies that it has full rank (over $\K[\sigma]$)
  which in turn implies that $\det \tilde{A}_0 \neq 0$ as a polynomial
  in $\K[\sigma]$.

  \smallskip

  It remains to reason about the termination of the algorithm. First,
  we note that the outer loop starting in \autoref[Step]{outer} cannot be
  reached infinitely often: A new \PNF\ is only computed if the lower
  rows of $\tilde{A}_0$ are not in the row space of the upper
  block. This implies that we either gain new nonzero rows in the
  \PNF\ or that the degrees or positions of the leading monomials
  decrease. In both cases, the module generated by the rows of
  $\tilde{A}_0$ becomes strictly larger which can happen only finitely
  often since $\RV{\K(t)[\sigma]}m$ is noetherian.

  Second, we have to show that the inner loop starting in
  \autoref[Step]{inner} cannot be run infinitely often. For this, assume
  that the matrix
  \begin{equation*}
    A =
    \begin{pmatrix}
      B\\
      C
    \end{pmatrix}
    \in \Mat{\K[t,\sigma]}mm
  \end{equation*}
  with $B \in \Mat{\K[t,\sigma]}rm$ and $C \in
  \Mat{\K[t,\sigma]}{(m-r)}m$ has full (row) rank and assume that the
  inner loop repeats infinitely, \ie, assume that for every $\nu \geq
  0$ there are $C_\nu \in \Mat{\K[t,\sigma]}{(m-r)}m$ and $X_\nu \in
  \Mat{\K[t,\sigma]}{(m-r)}{r}$ such that
  \begin{equation*}
    \begin{pmatrix}
      \ID[r] & \ZERO \\
      t^{-\nu} X_\nu & t^{-\nu} 
    \end{pmatrix}
    \begin{pmatrix}
      B \\
      C
    \end{pmatrix}
    =
    \begin{pmatrix}
      B \\
      C_\nu
    \end{pmatrix}
  \end{equation*}
  (where we write $t^{-\nu}$ instead of $t^{-\nu} \ID[m-r]$ for the
  lower left block). We explained why the transformation matrices
  always look like this in \autoref[Remark]{rem:trans} above.

  We now form the quotient (skew) field $\K(t,\sigma)$ of
  $\K[t,\sigma]$. Since $A$ has full row rank, it does have a
  (two-sided) inverse by \autoref[Remark]{rem:inverse}. Thus, the
  equation above is equivalent to the identity
  \begin{equation*}
    \begin{pmatrix}
      \ID & \ZERO \\
      t^{-\nu} X_\nu & t^{-\nu} 
    \end{pmatrix}
    =
    \begin{pmatrix}
      B \\
      C_\nu
    \end{pmatrix}
    A^{-1}
    =
    \begin{pmatrix}
      B P & B Q \\
      C_\nu P & C_\nu Q
    \end{pmatrix}
  \end{equation*}
  where we write
  \begin{math}
    A^{-1} =
    \begin{pmatrix}
      P & Q
    \end{pmatrix}
  \end{math}
  with $P \in \Mat{\K(t,\sigma)}mr$ and $Q \in
  \Mat{\K(t,\sigma)}m{(m-r)}$. In particular, we have the identity
  \begin{math}
    C_\nu Q = t^{-\nu}.
  \end{math}
  Note that $Q$ does not depend on $\nu$, \ie, the denominator in $t$
  of $Q$ is the same for every $\nu \geq 0$. In addition, $C_\nu$ is a
  polynomial matrix, \ie, its denominator is always $1$. Thus, the
  denominator on the \LHS\ is bounded. However, the denominator of the
  \RHS\ is not. This is a contradiction. Hence, there cannot be
  infinitely many steps where $C_\nu$ is in the row space of $B$.
\end{proof}

\medskip

An analogous method works for the $t$-leading matrix. The only
differences to \autoref[Algorithm]{alg:lambda} are that we have to
work with the \PNF\ of the $t$-leading matrix $\tilde{A}_\ell$ in
\autoref[Step]{inner'} and that we shift the lower rows by $t$ instead
of $t^{-1}$ in \autoref[Step]{inner'}. Again, we state the algorithm
without explicit computation of the transformation matrix.

\begin{algorithm}\label{alg:rho}~
    \begin{description}
  \item[Input:] A regular matrix $A \in \Mat{\K[t,\sigma]}mm$.
  \item[Output:] An equivalent matrix $B \in \Mat{\K[t,\sigma]}mm$
    such that the $t$-leading matrix of $B$ is regular.
  \item[Procedure:]
    \begin{enumerate}
    \item\label{outer'} Compute the $t$-leading matrix
      $\tilde{A}_\ell$ of $A$. If $\det \tilde{A}_\ell \neq 0$, then
      return $A$.
    \item Else, compute the \PNF\ of $\tilde{A}_\ell$ and let $X \in
      \MatGr{\K[\sigma]}m$ be the corresponding transformation matrix
      and let $r$ be the rank of $\tilde{A}_\ell$. Set $A \becomes X
      A$ and write the new trailing matrix of $A$ as
      \begin{equation*} \tilde{A}_\ell =
        \begin{pmatrix} \tilde{B}_\ell \\ \ZERO
        \end{pmatrix}
      \end{equation*} where $\tilde{B}_\ell \in \Mat{\K[\sigma]}rm$ is in
      \PNF.
    \item\label{inner'} Let $T = \diag(\ID[r], t \ID[m-r])$ and set $A
      \becomes T A$.\footnote{This shifts the lower $m-r$ rows of $A$
        by $t$.} The new trailing matrix is now
      \begin{equation*} \tilde{A}_\ell =
        \begin{pmatrix} \tilde{B}_\ell \\ \tilde{C}_\ell
        \end{pmatrix}
      \end{equation*} 
      with the same $\tilde{B}_\ell$ as before and some
      matrix $\tilde{C}_\ell \in \Mat{\K[\sigma]}{(m-r)}m$.
    \item Use \GBs\ reduction of $\tilde{C}_\ell$ by $\tilde{B}_\ell$
      trying to eliminate all the rows in $\tilde{C}_\ell$. Write the
      result as $\tilde{C}_\ell = X \tilde{B}_\ell + Y$ with a matrix
      $X \in \Mat{\K[\sigma]}rr$ and a matrix $Y \in
      \Mat{\K[\sigma]}{(m-r)}m$. Let
      \begin{equation*} Q =
        \begin{pmatrix} \ID[r] & \ZERO \\ -X & \ID[m-r]
        \end{pmatrix}
      \end{equation*} 
      and set $A \becomes Q A$.
    \item If $Y = \ZERO$, then continue the inner loop with
      \autoref[Step]{inner'}. Else, go back to the outer loop in
      \autoref[Step]{outer'}.
    \end{enumerate}
  \end{description}
\end{algorithm}

\begin{theorem}\label{thm:rho}
  \autoref[Algorithm]{alg:rho} is correct and terminates.
\end{theorem}
\begin{proof}
  The proof is mostly analogous to that of
  \autoref[Theorem]{thm:lambda}. The only noticeable change is that
  for the termination of the inner loop we have to deal with
  transformation matrices of the shape
  \begin{equation*}
    \begin{pmatrix}
      \ID & \ZERO \\
      t^\nu X_\nu & t^\nu
    \end{pmatrix}
  \end{equation*}
  (with some $X_\nu \in \Mat{\K[t,\sigma]}{(m-r)}r$) instead of having
  fractions in $t$. However, we come to an analogous equation $C_\nu Q
  = t^\nu$ where the degree of the \LHS\ is bounded while the degree
  of the \RHS\ is not. (It is easy to check that the degree in $t$ of
  $A$ during execution of \autoref[Algorithm]{alg:rho} is always equal
  to $\ell$). Thus, we arrive at a similar contradiction as for the
  proof of \autoref[Theorem]{thm:lambda}.
\end{proof}

\begin{example}
  We continue~\autoref[Example]{ex:deg}. The $t$-leading matrix of the
  system was
  \begin{equation*}
    \begin{pmatrix}
      0 & 0 \\
      0 & -64
    \end{pmatrix}
  \end{equation*}
  which is not regular. Thus, we will apply
  \autoref[Algorithm]{alg:rho}. The \PNF\ of the $t$-leading matrix is
  simply
  \begin{equation*}
    \begin{pmatrix}
      0 & 64 \\
      0 & 0
    \end{pmatrix}
    \qqtext{with transformation}
    X =
    \begin{pmatrix}
      0 & -1 \\
      1 & 0
    \end{pmatrix}.
  \end{equation*}
  (Actually, for the proper \PNF\ we would need to divide the first
  row by $64$; however, we want to avoid fractions in order to save
  some space.) We obtain the new system matrix
  \begin{equation*}
    \scriptsize
    A \becomes X A =
    \begin{pmatrix}
      -\sigma^2
      + (16 t^2 - 16 t + 12) \sigma
      + (-128 t^2 + 64 t - 32)
      &
      -8 \sigma
      + (64 t^4 + 8) 
      \\
      \sigma^2 
      + (-16 t + 4) \sigma
      + (128 t - 32)
      & 
      8 \sigma
      + (-64 t^3 - 8).
    \end{pmatrix}
  \end{equation*}
  Next, we multiply the lower row by $t$ which leads to the new
  $t$-leading matrix
  \begin{equation*}
    \begin{pmatrix}
      0 & 64 \\
      0 & -64
    \end{pmatrix}
  \end{equation*}
  which does not have full rank yet. \GB\ reduction amounts to adding
  the upper row to the lower. We obtain the new system matrix
  \begin{equation*}
    \scriptsize
    \begin{pmatrix}
      -\sigma^2
      + (16 t^2 - 16 t + 12) \sigma
      + (-128 t^2 + 64 t - 32)
      &
      -8 \sigma
      + (64 t^4 + 8) 
      \\
      (t - 1)\sigma^2 
      + (12 - 12 t) \sigma
      + (32 t - 32)
      &
      (8 t - 8) \sigma
      + (8 - 8 t)
    \end{pmatrix}.
  \end{equation*}
  It turns out that we have to shift the lower row by $t$ three times
  before the $t$-leading matrix changes again (since the lower row is
  only linear in $t$). The final result is the system matrix
  \begin{equation*}
    \scriptsize
    \begin{pmatrix}
      -\sigma^2
      + (16 t^2 - 16 t + 12) \sigma
      + (-128 t^2 + 64 t - 32)
      &
      -8 \sigma
      + (64 t^4 + 8) 
      \\
      (t^4 - t^3)\sigma^2 
      + (12 t^3 - 12 t^4) \sigma
      + (32 t^4 - 32 t^3)
      &
      (8 t^4 - 8 t^3) \sigma
      + (8 t^3 - 8 t^4)
    \end{pmatrix}
  \end{equation*}
  with $t$-leading matrix
  \begin{equation*}
    \begin{pmatrix}
      0 & 64 \\
      \sigma^2 - 12 \sigma + 32 & 8 \sigma - 8
    \end{pmatrix}.
  \end{equation*}
  This yields $\rho = -64 \sigma^2 + 768 \sigma - 2048$ with roots $4
  = q^2$ and $8 = q^3$. Thus, the degree of polynomial solutions is at
  most $3$. From \autoref[Example]{ex:bound} we do know that the solution space
  is spanned by $(t^3, 1)$ and $(t^2, 1)$; so the bound is actually
  sharp in this case.
\end{example}

\section{Conclusions}\label{sec:concl}

In this paper we have presented a method for determining degree bounds
for the polynomial solutions and for determining the power of $t$ in
the denominator bound for rational solutions of $q$-recurrence
systems. Although the translation from the scalar case seems to be
straight-forward at first, we quickly discovered a problem when
certain determinants are vanishing. This required us to develop a new
method for transforming the system into an equivalent form where those
determinants are non-zero.

There exist other ways of obtaining the same information, as for
instance uncoupling or transformation of the system into first
order. However, those approaches are often computationally costly
(since uncoupling is and since conversion to first order usually
introduces a lot of new variables). The method presented in this paper
does avoid those problems. 

We do not have strict bounds on the complexity of the transformation
algorithm yet; we hope to deliver those in the near future. However,
in all the examples which we have computed, termination usually
occurred within a very small number of steps. Moreover, the degree
bounds found were usually tight. Although we did not yet carry out
extensive comparisons with the existing methods mentioned above, we
are therefore confident that this method will prove to be useful in
practical applications.


\providecommand{\bysame}{\leavevmode\hbox to3em{\hrulefill}\thinspace}
\providecommand{\MR}{\relax\ifhmode\unskip\space\fi MR }
\providecommand{\MRhref}[2]{%
  \href{http://www.ams.org/mathscinet-getitem?mr=#1}{#2}
}
\providecommand{\href}[2]{#2}

\end{document}